\documentclass[12pt]{article}
\usepackage{amssymb}
\usepackage{amsthm}
\usepackage{amsmath}
\usepackage{amscd}
\usepackage[all,cmtip]{xy}
\usepackage[curve]{xypic}
\usepackage{hyperref}
\usepackage[hypcap]{caption}
\usepackage[latin2]{inputenc}
\usepackage{t1enc}
\usepackage{enumerate}
\usepackage[mathscr]{eucal}
\usepackage[british]{babel}
\usepackage[dvips]{graphicx}
\usepackage[dvips]{epsfig}
\usepackage{epsf}
\usepackage{indentfirst}
\usepackage{url}
\title{\textbf{The Grothendieck group of completed distribution algebras}}
\author{Tamas Csige}

\newcommand{\drho}{D_{r}(G,K)}
\newcommand{\drhoh}{D_{r}(H,K)}
\newcommand{\filg}{\textnormal{Fil}^0(D_r(G,K))}
\newcommand{\fillg}{\textnormal{Fil}^1(D_r(G,K))}
\newcommand{\gr}{\mathrm{gr}^{\cdot}}

\newcommand{\grr}{\textnormal{gr}^0}
\newcommand{\f}{\textnormal{Fil}}
\newtheorem{thm}{Theorem}[section]
\newtheorem{pro}[thm]{Proposition}
\newtheorem{lem}[thm]{Lemma}

\newtheorem{df}[thm]{Definition}

\newtheorem{rem}[thm]{Remark}

\theoremstyle{definition}
\begin{document}
\vspace{2cm}
\maketitle
\begin{abstract} \noindent Let $G$ be a compact $p$-adic analytic group with no element of order $p$ and let $H$ be the maximal uniform normal subgroup of $G$. Let $K$ be  a finite extention of $\mathbb{Q}_p$. We show that the Grothendieck group of the completion of the algebra of locally analytic distributions on $G$ is isomorphic to $\mathbb{Z}^c$  where $c$ is the number of conjugacy classes in $G/H$ relative prime to $p$, provided that $K$ is big enough. In addition we will see the algebra $K[[G]]$ of continuous distributions on $G$ has the same Grothendieck group.

\vspace{1cm} \noindent \textbf{Keywords}: Grothendieck group, Algebras of $p$-adic distributions, Iwasawa algebra
\newline \noindent \textbf{MSC}:18F30, 22E35, 20G05
\end{abstract}
\vspace{2cm}
\section{Introduction}
\noindent Let $p$ be a prime. Let $\mathbb{Q}_p \subseteq K$ be a finite extention of $\mathbb{Q}_p$. $K$ is a complete and discretely valued field and denote its residue field by $k$. Locally analytic representations of a locally $\mathbb{Q}_p$-analytic group $G$ were systematicaly studied by P. \ Schneider and J. \ Teitelbaum in \cite{S2}, \cite{S3}, \cite{S4}, \cite{S5}, \cite{S6}, \cite{S1}. In general, such a representation is given by a continuous action of $G$ on a locally convex topological vector space $V$ over a spherically complete extention $\mathbb{Q}_p \subseteq K$  such that the orbit maps $g \mapsto gv$ are locally analytic functions on $G$. We would like to mention that any finite field extention of $\mathbb{Q}_p$ is locally compact and hence spherically complete. On the other hand e.g. $\mathbb{C}_p$ is not spherically complete. In special cases the class of such representations includes many interesting examples, such as finite dimensional algebraic representations and smooth representations of Langlands theory. A reasonable theory of such representations requires the identification of a finiteness condition that is broad enough to include the important examples and yet restrictive enough to rule out pathologies. The appropriate finiteness condition is called "admissibility". The admissible locally analytic representations form an abelian category. 
\bigskip  \newline \noindent The way of characterization of such representations is via algebraic approach. If $V$ is in general a locally analytic representation, then its continuous strong dual $V_b^{'}$ with its strong topology becomes a module over the algebra $D(G,K)$. This algebra is the continuous dual of the locally $\mathbb{Q}_p$-analytic, $K$-valued functions on G, with multiplication given by convolution. When G is compact, $D(G,K)$ is a Fr\'echet-Stein algebra, but in general  neither noetherian nor commutative. The modules of the subcategory of the images of admissible representations are called coadmissible modules. These modules play crucial role on the theory of locally analytic representations. 
 \newline \noindent As mentioned above the distribution algebra is Fr\'echet-Stein, i.e. it is the projective limit of noetherian Banach algebras $D_r(G,K)$ where  these algebras are certain completions of $D(G,K)$ with respect to norms depending on $r \in p^\mathbb{Q}$ such that $1/p \leq r <1$ and the connecting maps are flat. Analogously, coadmissible modules are defined to be projective limits of certain finitely generated $\drho$-modules $M_r$. For details see section  \ref{dro}
 \newline \noindent In a recent paper \cite{Z} G. \ Z\'abr\'adi studied the Grothendieck group of  the completed distribution algebra $\drho$ in the special case when the group $G$ is a uniform pro-$p$ group and showed that $K_0(\drho) \simeq \mathbb{Z}$, i.e. every finitely generated projective module is stably free. In this paper we give a generalization and show the following theorem:
\begin{thm}\label{main}\textnormal{Le $G$ be a $p$-adic analytic group that has no element of order $p$ and let $H$ be its maximal uniform subgroup. If $K$ is large enough depending on $r$ and $G/H$ then the Grothendieck group of the ring $\drho$ is isomorphic to $\mathbb{Z}^c$ where $c$ denotes the number of conjugacy classes in $G/H$ of order relative prime to $p$.
}
\end{thm} \noindent How $K$ depends precisely on $r$ and $G/H$  is the following: On one hand if $r = p^{s/t}$ then we assume $\sqrt[t] {p} \in K$. On the other hand we require $K$ to be the splitting field of $G/H$ (for the definition of splitting field, see  \ref{split}). These assumptions are fairly mild, for example in the second case due to a theorem of Brauer, if $n$ is the exponent of $G/H$ and $K$ contains a primitive $n$th root of unity then $K$ is a splitting field for not just $G/H$ but for any subgroup.
\newline \noindent The  assumption on $G$ is also weak. For example if $p \geq 5$, $GL_2(\mathbb{Z}_p)$ has no element of order $p$. 
\newline \noindent  Another very interesting result will be proved as a side-effect: The algebra $K[[G]]$ of continuos distributions on $G$ which is just the localization of $\mathcal{O}_K[[G]]$ where $\mathcal{O}_K$ is the ring of integers in $K$ (when $K=\mathbb{Q}_p$ it is just the localization of the  Iwasawa algebra over $G$) has the same Grothendieck group. For the definition of $K[[G]]$ see Section \ref{iwa}.
 \newline \noindent In the next section we recall some facts and known results that we will use. Then  using the filtration induced by the norm dependig on $r$ we show that the Grothendieck group of the $0$th degree subring of the associated graded ring of $\drho$ is isomorphic to $\mathbb{Z}^c$, and also that $K_0(\grr) \simeq K_0(\filg)$. The next part will show the crucial fact that  $\filg$ has finite global dimension if $G$ has no element of order $p$. Using all these results and K-theoretic tools we prove that the $K_0$ of $K[[G]]$ is isomorphic to $\mathbb{Z}^c$ and then conclude the same for $\drho$.
\bigskip \newline \noindent \textbf{Acknowledgments.}  The author was supported by the Strategic Initiative Funding of  Humboldt-Universit\"at zu Berlin. \newline \noindent The author would like to thank to Gergely Z\'abr\'adi for his constant help and valuable comments.
and all the useful discussions.
\section{Preliminaries}\label{pre} 
\label{drho}\subsection{The algebra $D_{r}(G,K)$}\label{dro} \noindent Let $p$ be a prime and $p > 2$ (we exclude the case $p =2$ for simplicty reasons.) Let $\mathbb{Q}_p \subseteq K$ and $G$ be a locally $\mathbb{Q}_p$-analytic group. Recall from  \cite{S1} that when $H$ is uniform there is a bijective global chart \[ \mathbb{Z}_p^d \rightarrow G\] \[ (x_1, \dots,x_d) \mapsto (h_1^{x_1}, \dots, h_d^{x_d})\] where $h_1,\dots,h_d$ is a fixed minimal, ordered topological generating set of $G$. Putting $b_i: = h_i - 1 \in \mathbb{Z}[G]$, one can identify $D(G,K)$ with the ring of all formal series \[ \lambda = \sum d_\alpha b^\alpha \] where $\alpha$ is a multiindex ($\alpha = (\alpha_1,\dots,\alpha_n)$), $b^\alpha = b_1^{\alpha_1}\cdot \cdot \cdot b_d^{\alpha_d}$ and $d_\alpha \in K$ such that the set $\{ |d_\alpha| r^{|\alpha|}\}$ is bounded for any $0 < r < 1$.  When $r \in p^{\mathbb{Q}}$ there is a multiplicative norm $|| \cdot ||_r$ on $D(G,K)$ by Theorem $4.5$ in \cite{S1} given by \[ ||\lambda||_r := \textnormal{sup}_{\alpha} |d_\alpha|\rho^{|\alpha|}\] Whenever $G$ is an arbitrary compact $\mathbb{Q}_p$-analytic group $G$ has a uniform normal subgroup of finite index  (for details, see e.g. \cite{Dix} Theorem $8.32.$), thus we can choose coset representatives $g_1,\dots,g_n$, $n = G/H$ and see that $D(G,K)$ is a crossed product of $D(H,K)$ and the group $G/H$ (for details see \ref{cr}). Hence $D(G,K) = \bigoplus D(H,K)g_i$ (we abuse the notation and denote the Dirac distributions $\delta_{g_i}$ by $g_i$). \newline \noindent  Define norms \[q_r(\mu) : = \textnormal{max}(||\lambda_1||_r,\dots,||\lambda_n||_r)\] where $\mu \in D(G,K), \mu = \sum \lambda_i(\mu) g_i $. We need to show first that the norm is still submultiplicative given that $r \in p^{\mathbb{Q}}$.
\begin{pro}\label{norm}\textnormal{Let $G$ and $r$ be as above. The maximum norm is submultiplicative on $D(G,K)$. }
\end{pro}
\begin{proof} Let us take two elements $\mu_1 = \sum \lambda_i(\mu_1) g_i$ and  $\mu_2 = \sum \lambda_i(\mu_2) g_i$ from $D(G,K)$. Then \[ q_r(\mu_1\mu_2)=q_r (\sum_{i,j}\lambda_i(\mu_1)g_i\lambda_j(\mu_2)g_j)\]
We can write the expression above in the following form:
\[q_r (\sum_{i,j}\lambda_i(\mu_1)(g_i\lambda_j(\mu_2) g_i^{-1})g_ig_j)  \] The product $g_ig_j$ differs from a coset representative only by multiplication with an element $h_{i,j} \in H$, i.e. $g_ig_j = h_{i,j}g_k$. Hence it can be seen that the norm of the product is equal to the maximum of the values $||\sum_i \lambda_i(\mu_1) g_i\lambda_j(\mu_2)g_i^{-1}h_{i,j}||_r$. But we know that $||g_i\lambda_j(\mu_2)g_i^{-1}||_r = ||\lambda_j(\mu_2)||_r$. Moreover the norm is multiplicative and for any $h \in H$, $||h||_r = \prod ||h_i||_r^{x_i}$ but $||h_i||_r =||(h_i-1)+1||_r = \textnormal{max}\{1, ||h_i -1||_r\} = 1$ hence  \[q_r(\mu_1\mu_2) \leq \textnormal{max}(||\lambda_i(\mu_1)||_r||\lambda_j(\mu_2)||_r)= q_r(\mu_1)q_r(\mu_2)\]
\end{proof}
\noindent We denote the completion of $D(G,K)$ by $D_r(G,K)$. We recall the next result:
\begin{thm} \label{dglob} \textnormal{For any compact  locally $\mathbb{Q}_p$-analytic group $G$ the notherian Banach algebra $\drho$ is Auslander regular of global dimension $\leq d$. }
\end{thm}
\begin{proof}This is Theorem $8.9$ in \cite{S1}.
\end{proof}
\begin{rem} \textnormal{This result is valid in greater generality, namely when $G$ is locally $L$-analytic, where $L$ is a finite extention of $\mathbb{Q}_p$. For details see \cite{Sm}}.
\end{rem}
 \label{cr}\subsection{Group rings and generalizations} \noindent In this section following \cite{CM} (Part I. Chapter 5.) we briefly recall the basic notions of group rings and their generalizations and show some properties of distribution algebras and the rings connected to them that we will use later.
It is well known that given a ring $R$ and a group $G$ the group ring $R[G]$ is defined to be a free right $R$-module with elements of $G$ as a basis and with multiplication given by $(gr)(hs) = (gh)(rs)$ together with bilinearity. We could have defined it by some universal property, for details see \cite {CM}. 
\newline \noindent If we allow the group to have some action on the ring of scalars, more precisely, let $R$ be a ring, $G$ a group and $\varphi$ a homomorphism $\varphi:G \rightarrow \textnormal{Aut}R$. Let us denote the image of $r \in R$ under $\varphi(g)$ by $r^g$.  The skew group ring $R\#G$ is defined to be the free right $R$-module with elements of $G$ as a basis as before but the multiplication is defined by $(gr)(hs)=(gh)(r^gs)$. Evidently the skew group ring contains $G$ as a subgroup in its group of units, and $R$ as a subgring. If we put $\varphi(g) = 1$ we get back the usual group ring structure. \newline \noindent We need a more generalized notion which is called the crossed product. 
\begin{df} \textnormal{Let $R$ be a ring and $G$ a group. Let $S$ be a ring containing $R$  and a set of units $\overline{G} = \{\overline{g} | g \in G\}$ isomorphic as a set to $G$ such that \newline \noindent (i) $S$ is a free right $R$-module with basis $\overline{G}$ and $\overline{1}_G = 1_S$, \newline \noindent (ii) for all $g_1,g_2 \in G$ $\overline{g_1}R = R\overline{g_1}$ and $\overline{g_1} \ \overline{g_2}R=\overline{g_1g_2}R$
\newline \noindent Then $S$ is called a crossed product and we denote  such a ring by $R*G$.}
\end{df}
\subsection{ Completed group algebras}
\label{iwa} \noindent  Let $G$ be a compact $\mathbb{Q}_p$-analytic group, $H$ its maximal open uniform subgroup and let $K$ be any finite extention of $\mathbb{Q}_p$ with ring of integers $\mathcal{O}_K$; a finite extention of $\mathbb{Z}_p$. Fix a uniformizer element $\pi$ of $\mathcal{O}_K$ and let $k$ be the residue field of $\mathcal{O}_K$. We will denote by $\mathcal{O}_K[[G]]$ the completed group algebra of $G$ with coefficients in $\mathcal{O}_K$: \[ \mathcal{O}_K[[G]] = \varprojlim \mathcal{O}_K[G/N]\] where $N$ runs over all the open normal subgroups of $G$. Similarly one can define \[k[[G]] = \varprojlim k[G/H] \] If $G$ is finite, these are just the usualy group algebra. If $K = \mathbb{Q}_p$ then the first one is called the Iwasawa algebra of $G$. \noindent Many ring theoretic properties of completed group algebras are known, the following result is due to Brumer (\cite{BR} Theorem $4.1$) who computed the global dimension of the completed group algebra of an arbitrary profinite group $G$ with coefficients in a pseudo-compact ring $R$. As a consequence of his work, we have 
\begin{thm} \label{okok} \textnormal{ Let be a compact $p$-adic analytic group of dimension $d$. Then both $k[[G]]$ and $\mathcal{O}_K[[G]]$ have finite global dimension if and only if $G$ has no element of order $p$, and in this case \[ \textnormal{gld}(\mathcal{O}_K[[G]]) = d+1 \ \ \ \ \ \ \ \ \ \ \textnormal{gld}(k[[G]]) = d\]}
\end{thm}
\bigskip \noindent The localization of $\mathcal{O}_K[[G]]$ at $S=\{1,\pi, \pi^2, \dots\}$: \[ K[[G]] = K \otimes_{\mathcal{O}_K} \mathcal{O}_K[[G]] \]  The authors in \cite{S1} call the ring $K[[G]]$  the Iwasawa algebra of "measures", and it is dual to the $K$-valued continuous functions on $G$. Throughout the paper we will call it the algebra of continuous distributions.
\newline \noindent Fix a representing system of elements of the cosets in $G/H$, $g_1,\dots,g_n$. Let us consider the algebra $\drho$. The left action of $G/H$ via the dirac distributions of a fixed representation system of the cosets is induced by the action on the $b_i = h_i - 1$ which is the following: $g_j b_i = (g_j b_i g_j^{-1})g_j$. This gives us a homomorphism $\varphi :G \rightarrow \textnormal{Aut}(\drhoh)$ since the conjugation is a group automorhism of $H$ and the action is trivial on the elements of $K$. Moreover $\drho$ is a free right $\drhoh$-module on basis $\delta_1,\dots,\delta_{g_n}$. The only problem is that the map $\psi : G/H \rightarrow \drho$ of the representing system into the group of units of $\drho$ is just a morphism of sets. So $\drho = \drhoh* G/H$ just like in the case of Iwasawa algebras.
\label{f}\subsection{Filtration and grading}\noindent First recall the definition of a filtration on a ring and its associated graded ring.
\begin{df}\textnormal{Let $R$ be an associative unital ring. We say that $R$ is filtered if it is equiped with a family of additive subgroups of $R$, $(\f^iR)_{i \in \mathbb{R}}$, such that 
\begin{itemize}
\item{$\f^i \subseteq \f^j$ if $ j \leq i$,}
\item{$\f^i \cdot \f^j \subseteq \f^{i+j}$,}
\item {$1 \in \f^0,$ }
\item{$\cup_{i \in \mathbb{R}} \f^i = R$} 
\end{itemize}
We say that a filtration is complete if $R \simeq \varprojlim R/\f^i$.}
\end{df}
\begin{rem}
\textnormal{Most of the books define a filtration to be an increasing filtartion. We defined it to be descending because our filtration on $\drho$ will be such, but it produces no difficulty. \noindent One can similarly talk about $\mathbb{N}, \mathbb{Z}$ filtrations depending on the index set. For more details on filtrations generaly see e.g. \cite{Hiu}.}
\end{rem}
 \noindent Let us recall one more fact about $\mathbb{N}$-filtered rings.
\begin{pro}
\textnormal{Let $R$ be a filtered ring with increasing $\mathbb{N}$-filtration. Then: \newline \noindent (i) rgld $R \leq$ rgld $\gr(R)$, \newline \noindent (ii) if $M$ is a filtered $R$-module, then pd $M \leq \textnormal{pd}_{\gr(R)}\gr(M)$ }
\end{pro}
\begin{proof} See \cite{CM} Theorem $7.6.17.$ and Corollary $7.6.18.$
\end{proof}
\noindent By Lemma \ref{norm} we can define a filtration on $D_r(G,K)$ (induced by the norm) the same way as in the uniform case: for any $s \in \mathbb{R}$ \[\textnormal{Fil}^s(D_r(G,K)) : = \{ \mu \in D_r(G,K) : q_r(\mu) \leq p^{-s}\}\] \[\textnormal{Fil}^{s+}(D_r(G,K)) : = \{ \mu \in D_r(G,K) : q_r(\mu) < p^{-s}\}\]
 For any $s \in \mathbb{R}$ define \[\gr(D_r(G,K)) := \oplus_s \f^s D_r(G,K)/\f^{+s}D_r(G,K) \] The ring $\gr(R)$ is called the associated graded ring of the filtered ring $R$.
When $G$ is uniform, it was shown in \cite{S1} Theorem $4.5$ that the associated graded ring of $\drho$ is isomorphic to $k[x_0,x_0^{-1}][x_1,\dots,x_d]$. The isomorphism is induced by the mapping $\sigma(b_i) \mapsto x_i$ where $\sigma(b_i)$ are the principal symbols of $b_i$. The graded ring of $K$ is just the Laurent polynomial ring $k[x_0,x_0^{-1}]$ and the variable $x_0$ is the image of the principal symbol of the uniformizer element. The associated graded ring can be computed in the general case:
\begin{lem}\textnormal{Let $G$ be a $\mathbb{Q}_p$ analytic group and $r \in p^{\mathbb{Q}}$. Then $\gr(D_r(G,K))$ is isomorphic to the skew group ring $ \gr(\drhoh)\#G/H$.}
\end{lem}
\begin{proof} Indeed, in the distribution algebra, the problem is that when we multiply two elements $g_i,g_j$ of the representing system of cosets in $G/H$, $g_ig_j$ not necessarily a representing element, it is in  $G$ and of the form $g_ig_jh$ for some $h \in H$. We show that it is not a problem any more in the associated graded ring \[||g_ig_jh - g_ig_j||_\rho = || (g_ig_j)(h-1)||_\rho = ||g_i||_\rho||g_j||_\rho||h-1||_\rho\] If $h=h_i$ then using that $G$ is a saturated $p$-valued group with valuation $\omega$ we see that $||h-1||_\rho \leq \rho^{\omega(h_i)} < 1$ and if $h = h_1^{\alpha_1}\cdot \cdot \cdot h_d^{\alpha_d}$, $\sum \alpha_i \leq 2$ then by the multiplicity of the norm shows that the norm of $(h-1)$ is still $<1$. So $\||g_ig_jh - g_ig_j||_\rho <1$ which shows that in the associated graded ring all the elements in one particular coset are mapped to the same principal symbol. 
\end{proof}
\noindent So in particular \[ \filg = \bigoplus  \f^0(\drhoh)g_i\] and \[ \gr(\drho) = \bigoplus \gr(\drhoh)g_i\] 
\label{k0} \subsection{The Grothendieck group of rings and categories} \noindent  We recall the definitions and results we use in the  proofs.
\begin{df}\textnormal{ Let $\mathcal{A}$ be an Abelian-category. Its Grothendieck group $K_0(\mathcal{A})$ is the abelian group having one generator $[A]$ for each object in $\mathcal{A}$ and a realtion $[A] = [A_1]+[A_2]$ for every short exact sequence \[0 \rightarrow A_1 \rightarrow A \rightarrow A_2 \rightarrow 0 \] in $\mathcal{A}$ }
\end{df}
\begin{rem}\textnormal{ We get back the classical definition of $K_0$ of a noetherian ring $R$ (for details see e.g. \cite{B} Chapter II.) if we consider $\mathcal{A}$ to be the full subcategory of finitely generated projecive modules. If we consider the category of finitely generated modules over  R, we denote its Grothendieck group by $G_0(R)$ and by $K_0(R)$ when we only consider the finitely generated projectives over $R$ (i.e. the classical $K_0$ of a ring)}
\end{rem}
\begin{pro}\label{idem}\textnormal{Let $R$ be a ring and $I$ a nilpotent, or more generally a complete ideal in $R$ (i.e. $R$ is an $I$-adic ring). Then \[K_0(R/I) \simeq K_0(R)\]}
\end{pro}
\begin{proof} It is Lemma $2.2.$ in \cite{B} Chapter II.
\end{proof}
\begin{thm}(Devissage Theorem) \textnormal{ Let $\mathcal{B} \subset \mathcal{A}$ small albelian categories. Suppose that \newline (i) $\mathcal{B}$ is an exact abelian subcategory of $\mathcal{A}$, closed in $\mathcal{A}$ under subobjects and quotients, \newline \noindent (ii) Every object A of $\mathcal{A}$ has a finite filtration \[A = A_0 \supset A_1 \supset \dots \supset A_n = 0\] with all quotients $A_i/A_{i+1}$ in $\mathcal{B}$. \newline \noindent Then the inclusion functor $\mathcal{B} \subset \mathcal{A}$ is exact and induces an isomorphism \[K_0(\mathcal{B}) \simeq K_0(\mathcal{A})\]
}
\end{thm}
\begin{proof} See \cite{B} Chapter II.,Theorem $6.3.$
\end{proof}
\noindent Let $\mathcal{A}$ be again an abelian category. We call a subcateory  $\mathcal{B} \subset \mathcal{A}$ Serre-subcategory if $0 \to A \to B \to C \to 0$ is an exact sequence in $\mathcal{A}$, then $B \in \mathcal{B}$ if and only if $A, C \in \mathcal{B}$. It is well known that if $\mathcal{B}$ is a Serre-subcategory. we can form a quotient category $\mathcal{A}/\mathcal{B}$. (For details, see e.g. \cite{B} page $51.$)
\begin{thm}(Localization theorem) \label{loc}\textnormal{ Let $\mathcal{A}$ be a small abelian category, and $\mathcal{B}$ a Serre subcategory of $\mathcal{A}$. Then the following seuence is exact: \[ K_0(\mathcal{B}) \to K_0(\mathcal{A}) \to K_0(\mathcal{A}/\mathcal{B}) \to 0\]
}
\end{thm}
\begin{proof} See \cite{B} Chapter II., Theorem $6.4.$
\end{proof}
\noindent We will use a special case of the previous result: Let $R$ be a noetherian  ring, $x \in R$ a central non-zero divisior and $S =\{1,x,x^2,\dots,\}$ the central multiplicative set. The subcategory of finitely generated $S$-torsion modules, denoted by $x$-tors is a Serre subcategory in ($\textnormal{mod}-R$), the category of finitely generated right $R$-modules. There is a natural equivalence between ($\textnormal{mod}-S^{-1}R)$ and the quotient category $(\textnormal{mod}-R)/(x-\textnormal{tors})$. Then by the Localization Theorem, we have the following exact sequence: \[ K_0(x-\textnormal{tors}) \to G_0(R) \to G_0(R[1/x]) \to 0 \]
\begin{thm}\label{fund}(Fundamental Theorem for $G_0$-theory of rings)\textnormal{ Let $R$ be a noetherian ring, the inclusions $R \hookrightarrow R[t] \hookrightarrow R[t^{-1},t]$ induce isomorphisms \[G_0(R) \simeq G_0(R[t]) \simeq G_0(R[t,t^{-1}])\]}
\end{thm}
\begin{rem}\label{global} \textnormal{If $R$ is a noetherian ring with finite global dimension, there is another result called Fundamental Theorem for $K_0$ of regular rings. One of the  statements of the theorem is that if $R$ is a relguar noetherian ring then $G_0(R) \simeq K_0(R)$. For details see Theorem $7.8.$ in \cite{B} Chapter II.} 
\end{rem}
\section{$K_0$ of the 0th graded subring of $\gr \drho$} \label{3}
\noindent In this section we prove the following theorem and we will also make use of the technique of the proof later.
 \begin{thm}\label{grr}\textnormal
{The Grothendieck group of $\grr D_{\rho}(G,K)$ is isomorphic to $\mathbb{Z}^c.$
}
\end{thm}
\begin{proof}
\noindent We begin with two important lemmas.
\begin{lem}
\label{gr}\textnormal{$K_0(\grr K[G/H])$ is a direct summand of $K_0(\grr \drho)$.}
\end{lem}
\begin{proof}
\textnormal{ We have a surjective map of filtered rings \[ \varphi:\drho \rightarrow K[G/H] \] (with the induced filtration on $K[G/H]$) which maps and element $\xi = \sum g_i \lambda_i \in \drho$ to $\sum g_i k_0^{\lambda_i} \in K[G/H]$ where $k_0^{\lambda_i}$ is the constant term of the power series $\lambda_i$. This mapping induces a surjective homomorphism of graded rings
\[ \grr \varphi: \grr \drho \rightarrow  \grr K[G/H]\] the ring $\gr K[G/H]$ is isomorphic to $k [G/H]$. We also have a natural injection \[ \grr \psi: \grr K[G/H] \hookrightarrow \grr \drho\] It is easy to see that the composition of the two maps $\grr \varphi \circ \grr \psi = id_{\grr K[G/H]}$. We have seen in subsection $2.3$ that $K_0(-)$ is a functor from the category of rings to the category of Abelian groups so it takes identity to identity and that factorizes through the Grothendieck group of $\grr \drho$ but that means that $K_0(\grr K[G/H])$ is a direct summand of $K_0(\grr \drho)$ which is exactly the statement of the lemma.}
\end{proof}
\bigskip \noindent Let us denote by $\Phi$ and $\Psi$ the extention of scalars with respect to $\grr \varphi$ and $\grr \psi$. For example if we extend a $\grr \drho$-module $M$ then $M \otimes_{\grr \varphi} \grr K[G/H] \simeq M/MI$ where $I$ denotes the kernel of the morphism $\grr \varphi$. Meggondolni: $\Psi \circ \Phi$ projektivet projektivbe visz. 
 \begin{lem}\label{e}
\textnormal{ $K_0(\grr \drho)$ can be embedded into $K_0(\grr K[G/H])$.}
\end{lem}
\begin{proof} Its enough to show, that for a projective $\grr \drho$-module $P$ we have an isomorphism \[P \simeq P \otimes_{\grr \drho} \grr K[G/H] \otimes_{\grr K[G/H]} \grr \drho = \Phi(\Psi(P))\] First observe that $\Psi(P) = P/PI$ is projective as a $\grr K[G/H]$-module (since it will be a direct summand of a free module) and there is an injective $\grr K[G/H]$-module homomorphism into $P$ (since $P$ is a $\grr K[G/H]$-module in a natural way). Indeed we have the following  diagram

\bigskip
\hspace{4.5cm}
\xymatrix{
P \ar@{->>}[r]^{\alpha}
&P/PI \ar[d]^{id} \ar[r]
&0 \\
&P/PI \ar@{.>}[ul]_{\mu}}

\bigskip \noindent The  map $\mu$ is injective since this is a commutative diagramm. We have another $S$-module homomorphism for any S-module induced from an $R$-module $M$ by scalar extention \[\gamma_M:   M \otimes_{\grr K[G/H]} \grr \drho \rightarrow M \otimes_{\grr \drho} \grr \drho = M \] We can put these together  (denote $\gamma = \gamma_P$) and get an $R$-module homomorphism 
\label{eta} \[ \begin{CD} \label{1} \eta: P_1 = \Phi(\Psi(P)) @>\mu \otimes 1>> P \otimes_{\grr \drho} \grr K[G/H] @> \gamma \otimes 1 >> P 
\end{CD} \]
If we show that this is an isomorphism then we are done. Injectivity follows if we carefuly investigate the morphism. 
For surjectivity consider the subring $S=k[x_1,\dots,x_n]$ of $R = k[x_1,\dots,x_n][G/H]$. It is indeed a subring since $R$ is a skew group ring over $S$. 
Now it is easy to see that both \[ \Phi(\Psi(P)) = P/PI \otimes_{\grr K[G/H]} \gr \drho = P_1\] and $P$ are projective $S$-modules. Indeed it follows from the fact that both are projective over $R$ so they are direct summand of a finitely generated free $R$-module and $R$ is finitely generated free over $S$. We can also naturaly consider the map $\eta$ as $S$-module homomorphism.
We can put an $\mathbb{N}$-filtration and grading onto $S$ by the degree function and we can extend this filtration and grading onto $R$ by taking the maximum of the degrees of the coefficients of an element of $R$. In other words if we write the elements of $R$ as \[\lambda = \sum_{g_i \in G/H} p_i(\underline{x})g_i\] ($\underline{x} $ is a multivariable) \[\textnormal{Fil}^i(R) =\{ \lambda \in R \ | \ \textnormal{max}( \textnormal{deg}(p_i(\underline{x})) \leq i)\] So $R$ is a filtered and graded ring and the degree $0$th subring generates $R$ as an $S$-module. Since both $P_1$, $P$ are direct summands of $R^{n_i}$ ($i=\{1,2 \}$ and $n_i \in \mathbb{N}$), $P$ naturaly inherits the filtration and the grading and it is still true that $gr^0(P)$ generates $P$ as an $S$-module. One can see that  $\eta$ is a filtered and graded homomorphism, so it maps $gr^0(P_1)$ into $gr^0(P)$ . It is also true that they remain finitely generated over $S$ since $G/H$ is a finite group, so $R$ is a finitely generated $S$-algebra. One can now  use  the Theorem of Quilen-Suslin on projective modules over polynomial rings (for details see e.g. \cite{Qui} ), and see that both $P$ and $P/PI \otimes_{S} R=\Phi(\Psi(P))$ are graded-free modules of the same rank as $S$-modules and it is also true for the degree $0$ direct summands. Now let $P = <m_1,\dots,m_k>$ be a generating system of $P$. Then we see that $\overline{m_1} \otimes 1,\dots, \overline{m_k}\otimes 1$ generates $P_1$ as an $R$-module and now consider the degree $0$ submodules of $P$ and $P_1$. Observe that both $m_i$ and $\overline{m_i \otimes 1}$ are in the degree $0$ part of $P$ and $P_1$ respectively and nothing more, when we consider them as graded $S$-modules. If we follow the images of these elements we see that $\eta$ maps $\overline{m_i}\otimes 1$ to $m_i$. So the map in $\ref{eta}$ must be the natural isomorphism (of finitely generated free modules) on the degree $0$ part as $S$-modules, but it means that $P$ and $P_1$ are isomorphic as graded $S$-modules since the degree $0$ parts generate them but then they are isomorphic with the map $\eta$ as $R$-modules since it is still remains surjective.
\end{proof}
\begin{df} \label{split} \textnormal{Let $A$ be a finite group. A field $F$ is called splitting field for $A$  if, for any somple $F[G]$-module $V$, we have End$_{F[G]}(V) = F$.} 
\end{df}
\begin{lem}\label{k} \textnormal{The Grothendieck group of $k[G/H]$ is $\mathbb{Z}^c$ where $c$ is  the number of conjugacy classes relative prime to $p$ in $k[G/H]$.}
\end{lem}
\begin{proof} It follows from known facts in modular representation theory. By Corollary $3.2.4.$ in \cite{Mod} that the number of non-isomorphic classes of simple modules is equal to the number of $p$-regular conjugacy classes, i.e. the classes with order relative prime to $p$ (this is where one needs $K$ to be a splitting field of $G/H$).  Also $G/H$ is finite, so $k[G/H]$ is a finite dimensional $k$-algebra, and there is a  one-to-one correspondence between the indecomposable projective modules and the simple modules (for details, see e.g. \cite{Len} Theorem $7.1.$). Every finitely generated projective module can be decomposed into indecomposable projectives since they are direct summands of $k[G/H]^s$, for some $s \in \mathbb{N}$ and by the definition of the Grothendieck group, we see that $K_0(k[G/H]) \simeq \mathcal{Z}^c$.
\end{proof}
\noindent To finish the proof of Theorem \ref{grr} we need to use Lemma \ref{gr}, \ref{k} and \ref{e}.
\end{proof} 
\label{fil} \section{The global dimension of $\filg$}
\noindent From now we make an additional assumption on $K$. Namely if  $r = p^{a/b}$ then let $K$ be big enough that $\sqrt[b]{p} \in K$.
\noindent In this section we prove that the global dimension of $\filg$ is finite. Moreover it is equal to the global dimension of $\drho$. Let $R = \filg$ and $S = \drho$. We have the natural inclusion $R \hookrightarrow S$. In fact $S$ is just the localization of $R$ by the central, non-zero divisor $p$. It is easy to see that on $\filg$ the induced filtration is the following \[ \dots \subset  \textnormal{Fil}^2 \drho \subset \textnormal{Fil}^1\drho\subset \filg = \filg= \dots \] and the associated graded ring is just an $\mathbb{N}$-graded ring, i.e. \[\gr \filg = \grr \filg \oplus I\] where $I=\oplus_{n=1}^{\infty} \textnormal{gr}^n \filg$ is the augmentation ideal of $\gr \filg$. We have the usual projection and inclusion maps: $\pi: \gr \filg \to \grr \filg$ and $i: \grr \filg \hookrightarrow \gr \filg$. 
 \begin{df}
\textnormal{ Let $R$ be a filtered ring with filtration $FR$. We call a filtration on an $R$-module $M$ good, if there exist $m_1,\dots,m_n \in M$ and $k_1,\dots,k_n \in \mathbb{Z}$ such that for all $s \in \mathbb{Z}$ \[F^sM = \sum_{i=1}^n F^{s-k_i}R m_i \] }
\end{df}
\begin{rem} \textnormal{It can be shown for any filtered $R$-module that if $FM$ a good filtration then $\gr M$ is finitely generated over $\gr R$. Moreover since we have a Zariskian filtration on $\filg$ it is also true that good filtrations induce good filtrations on submodules and good filtrations are separated. For details see \cite{Hiu} Chapter I. section 5. and Chapter II.}
\end{rem}
\noindent Let $R$ be a filtered ring. We denote the category of finitely generated filtered modules over $R$ with 0 degree filtered homomorphisms (or simply filtered homomorphisms) $R$-Filt and $\gr R$-gr the category of finitely generated graded modules over the associated graded ring with degree 0 graded homomorphisms.  We also recall briefly what the projective objects in these two categories.
\begin{df}\textnormal{An object $L \in R$-filt is called filtered-free if it equipped with a good filtration such that the generators are free generators, i.e. there exist $e_1,\dots,e_n$ such that $F^sL = \oplus F^{s-k_i}Re_i$ and $e_i \in F^{k_i}-F^{k_i+1}$. }
\end{df}
\begin{lem}
\label{51} \textnormal{Let $M$ be a submodule of a finitely generated filtered-free module over $R$ (which means that it is an admissible module over $R$). Then $\gr_R(M) \otimes_{\gr R } \gr S \simeq \gr_S (M \otimes_R S)$.}
\end{lem} 
\begin{proof}
First of all, we know that $\gr S$ is a flat $\gr R$-module and both are noetherian rings. So by flatness property and the assumption on $M$, one can see that we have the following exact sequences \[ 0 \to \gr (M \otimes S) \hookrightarrow \gr (F^n \otimes S)\] \[ 0 \to \gr M \otimes \gr S \hookrightarrow \gr F^n \otimes \gr S \]   where $F^n$ is some finitely generated filtered-free  module over $R$ and $M \subset F^n$ as filtered module. Now if $N,L$ any two filtered modules over a filtered ring we have a graded-epimorphism  $\xi:\gr N \otimes \gr L \to \gr (N \otimes L)$ definied by $x_{(s)}=\sigma (x) \otimes \sigma (y)=y_{(t)} \mapsto (x \otimes y)_{s+t}$ for $x \in F_s(M)-F_{s-1}(N)$ and $y \in F_t(L) - F_{t-1}(L)$. So we have the following commutative diagram \[ \begin{CD} 
0 @>>> \gr(M \otimes S) @>>> \gr(F^n \otimes S)\\
@. @A \xi AA @A \xi AA \\
0 @>>> \gr M \otimes \gr S @>>> \gr F^n \otimes \gr S\\
\end{CD}
\]
Since $\xi$ is an epimorhism, we only need injectivity. Now one can see that the modules on the right hand side are isomorphic, since $F^n$ is filtered-free.  It follows that the map defined above between $\gr M \otimes \gr S$ and $\gr (M \otimes S)$ is also injective. Hence we are done.
\end{proof}
\begin{rem}\label{53} \textnormal{We know that $\gr S = \gr D_r(G,K) \simeq \grr R \otimes_k \gr K = \grr D_r(G,K) \otimes_k \gr K$. We also know that $\gr K = k[x_0,x_0^{-1}]$ so over $k$ it is a free module.  Let us assume that we have an exact sequence of $\grr R$-modules \[ 0 \to K \to N \to L \to 0 \]  the scalar extention by $\gr S$ is faithfuly flat since  $M \otimes_{\gr R} (\gr R \otimes_k \gr K) = (M \otimes_{\grr R} \grr R) \otimes_k \gr K \simeq M \otimes_k \gr K$ for any $\grr R$-module $M$ and $\gr K$ is a free module over $k$ hence tensoring with it is faithfuly flat.}
\end{rem}
\begin{lem}\label{52} \textnormal{Let $M$ be an admissible filtered module over $R$. Consider the functor $Q(M) = M/MI = M \otimes_{\gr R} \grr R$. Then $ Q(M) \otimes_{\grr R} \gr S \simeq \gr M \otimes_{\gr R} \gr S$.}
\end{lem} 
\begin{proof}  For such a module $M$, the induced filtration by the fitlered-free module is again good, so the $n$th direct summand of associated graded module has the following form \[ G(M)_n = \sum G(R)_{n+k_i}  \sigma (u_i)\] \[\grr (M) = G(M)_0 = \sum G(R)_{k_i}  \sigma (u_i)\] By the assumption on $K$, the possible values of $|| \ ||$ are non-negative  powers of $||\pi||$, the uniformizer element, hence for this special kind of grading $G_0(M) \otimes_k k[x_0] \simeq G(M)$ and since $\gr R \simeq k[x_0,x_1,\dots,x_d]\# G/H, \gr S  \simeq k[x_0,x_0^{-1}[x_1,\dots,x_d]\#G/H$, moreover $x_0$ commutes with every element in these rings,  we have the following $G(M) \otimes_{k[x_0]} k[x_0,x_0^{-1}] \simeq G(M) \otimes_{\gr R} \gr S$. But $G_0 \otimes_{\grr R} \gr K \simeq G_0 \otimes_k k[x_0,x_0^{-1}]   \simeq G_0 \otimes_k k[x_0] \otimes_{k[x_0]} k[x_0,x_0^{-1}]$ Hence we are done.
\end{proof}
\begin{thm}\label{dim}
\textnormal
{The global dimension of the noetherian algebra $\filg$ is finite and equals to  gl.dim.$D_r(G,K) +1$.
}
\end{thm}
\begin{proof}
It is enough to find a finite projective resolution (such that the smallest upper bound for such resolutions is d) for right ideals $J \subseteq R$. We know that $J$ has a natural good filtration, which is the induced filtration by $R$. We can take a filtered-free resolution of $J$ with respect to this filtration. Let us assume that it is infinite. All the syzygys are also admissible modules and admit good filtrations induced by the corresponding filtered-free modules. Tensor this resolution by $S$ over $R$. By flatness property we get a filtered-free resolution of $J \otimes S$. Since $S$ is Auslander-regular ring with global dimension $d$, there is a syzygy $M_n \otimes S$ such that $n \leq d$ and it is a projective $S$-module, hence filtered-projective. Now apply $\gr ()$ and use the flatness property of the functor to get a graded-free resolution of $\gr (M \otimes S)$ with a graded-projective syzygy $\gr (M_n \otimes S)$. So we have the following finite graded-projective resolution \[ 0 \to \gr (M_n \otimes S) \to \dots \to \gr (F^0 \otimes S) \to \gr (J \otimes S) \to 0\] Using Lemma \ref{51} we see that all the modules in the resolution come from the modules in the initial resolution after scalar extention of their associated graded module over $\gr R$, i.e. we have \[ 0 \to \gr M_n \otimes \gr S \to \dots \gr F^0 \otimes \gr S \to \gr J \otimes \gr S \to 0 \] Now using Lemma \ref{52} we see that the syzygy $\gr M_n \otimes \gr S$ comes from the graded $0$ part of $M_n$ (as the image with respsect to the projection) after scalar extention (it is true for all the syzygys, but we only need the projective one). By Remark \ref{53} we know that scalar extention is faithfuly flat and now it follows that $Q(M_n) = \grr M_n$ (We abuse this notation) is graded-projective over $\grr R$. We have the natural projection $M_n \to \grr M_n = M_n/M_nI$ and since the latter is projective (it is projective as an $\gr R$-module since the action of $\gr R$ is via the projection $\gr R \to \grr R$), it induces a section $\gamma:\grr M_n = M_n/M_nI \to \gr M_n$, so $\gamma \circ \pi = id$. We show that $M_n/M_nI \otimes_{\grr R} \gr R \simeq \gr M_n$. First we have a map $\nu: M_n/M_nI \otimes_{\grr R} \gr R \to \gr M_n \otimes_{\grr R} \gr R \to \gr M$. The first map is $\gamma \otimes 1$, the second is just the multiplication map. If we apply the functor $Q(-)=(- \otimes_{\gr R} \grr R)$ again, we get the following \[ Q(\nu): M_n/M_nI \otimes_{\grr R} \gr R \otimes_{\gr R} \grr R \to \gr M_n \otimes_{\grr R} \gr R \otimes_{\gr R} \grr R \to \gr M \otimes_{\gr R} \grr R \] One can see that the module on the left is isomorphic naturally to $Q(\gr M_n)$, the middle one is just $\gr M_n$ and the module on the right is again isomorphic to $Q(\gr M_n)$. The maps are now the following \[\begin{CD} Q(\gr M_n)  @>\gamma >> \gr M_n @> \pi >> Q(\gr M_n) 
\end{CD} \] We know that $\gamma \circ \pi = id$, hence $Q(\nu)$ is a graded isomorphism. Let us assume that $\nu$ has a non-trivial cokernel. Then since $Q$ is right exact, $Q(coker\nu) \neq 0$ but we know for any non-zero graded module $N$ that $Q(N) \neq 0$. Hence the above assumption is contradicting to the fact that $Q(\nu)$ is an isomorphism. So at this point we have a surjective map $M_n/M_nI \otimes_{\grr R} \gr R \to \gr M_n$. Let us now assume that $\nu$ has a non-trivial kernel. $Q(M_n)$ is projective by the argument above, hence $M_n/M_nI \otimes_{\grr R} \gr R$ is also projective over $\gr R$ and we have the following \[ \begin{CD} 0 @>>> ker\nu @>>> M_n/M_nI \otimes_{\grr R} \gr R @> \nu >> \gr M_n @>>> 0 \\
@. @. @V \mu VV @V \pi VV  \\
@. @. \grr M_n @>=>> \grr M_n \\
@. @. @VVV @VVV\\
@. @. 0@.0\\
\end{CD} \]
 Both $\nu$ and $\pi$ are surjective, so $\mu: = \pi \circ \nu$ is also surjective and again since we saw that $Q(M_n) = \grr M_n$ is projective, we have a section $\varphi: \grr M_n \to M_n/M_nI \otimes_{\grr R} \gr R$ such that $\varphi\circ \mu = id$. But then $(\varphi \circ  \pi) \circ \nu = id$ so we have constructed a section, meaning that we have the following exact sequence \[ \begin{CD} 0 @<<< ker \nu @<<< M_n/M_nI \otimes_{\grr R} \gr R @<<< \gr M_n @<<< 0 \end{CD}\] Applying $Q(-)$ again and using the right exact proerty we get that \[ M_n/M_nI \otimes_{\grr R} \gr R \to ker \nu \to 0 \] is exact and we know again the fact that if a graded module $N$ is not zero, then $Q(N) \neq 0$, so since $Q(\nu)$ is an isomorphism, we get that $ker \nu =0$. It means that $\gr M_n$ is graded projective and for the filtration on $\filg$ is separated, exhaustive and complete, we can use Proposition I.$ 7.2.1$ in \cite{Hiu} to see that $M_n$ is filtered-projective. Hence we showed that there is a finite non-negative integer $n$ such that the $n$th syzygy in the filtered-free resolution is filtered-projective and it is a finite projective resolution of $J$ after applying the natural forgetful functor. Of course it is bounded above by $d$. Since $S$ is just the localization of $R$ by a central non-zero divisor, standard results show that the upper bound is in fact an equality.   Now let us assume that we have a short exact sequence \[ 0\to J \to R \to M \to 0\]  where $M$ is a cyclic module and $J$ its annihilator (which is a right ideal). By the well known relationship of the projective dimension function and short exact sequences (for details see e.g. $7.1.6.$ in \cite{CM}) the pr$_R(M)$ = pd$_R(J)+1$. It is also well known that its enough to compute the projective dimension of cyclic modules.
\end{proof}
\section{The Grothendieck group of $D_{\rho}(G,K)$}
\noindent In this section we prove first that the Grothendieck group of $\drho$ is the same as of the 0th graded subring of $\gr \drho$ if $G$ is is a pro-$p$ $p$-adic analytic group. It is still interesting and requires no additional argument, apart from using some well known results.  First recall from \ref{f} that when $G$ is uniform, we know that the ring $\grr \drho = \filg / \fillg \simeq k[x_1,\dots,x_d]$ and moreover when $G$ is any $p$-adic analytic group the quotient is the skew group ring $k[x_1,\dots,x_d]\#G/H$ where $H$ denotes the maximal uniform open normal subgroup, the associated graded ring is just $k[x_0,x_0^{-1}][x_1,\dots, x_d]$ and $k[x_0,x_0^{-1}][x_1,\dots,x_d]\#G/H$ in the uniform and general case respectively. 
We proved in the last section that the Grothendieck group of $\grr \drho \simeq \mathbb{Z}^c$.
\begin{lem}\label{filk}\textnormal{The Grothendieck group of $\filg$ is isomorphic to $\mathbb{Z}^c$.}
\end{lem}
\begin{proof} It is easy to see that the induced filtration on $\filg$ that we saw at the begining of Section \ref{fil} and the filtration induced by $\textnormal{Fil}^1(\drho)$ are cofinal. Hence $\filg$ is an -adic ring. We showed in Theorem \ref{grr} that $\grr = \filg / \fillg$ has Grothendieck group isomorphic to $\mathbb{Z}^c$.  Using Proposition \ref{idem} we know that by idempotent lifting we can lift the $K_0$ of the quotient $\filg /\fillg$ to $\filg$ and we get an isomorphism.
\end{proof}
\begin{thm} \label{4a}\textnormal{Let $G$ be as above, but assume in addition that it is pro-$p$. The Grothendieck group of $\drho$ is isomorphic to $\mathbb{Z}$.}
\end{thm}
\begin{proof} Now $\drho$ is the localization of $\filg$ by the central regular element $p$. We use \ref{loc} to see that we hve the following exact sequence: \[K_0(tors-p) \to G_0(\filg) \to G_0(\drho) \to 0\] Since both $\filg$ and $\drho$ have finite global dimension we have a surjective map $\varphi: \mathbb{Z} \simeq K_0(\filg) \to K_0(\drho)$. Since $\drho$ has the invariant basis property, we have an injective map $\mathbb{Z} \hookrightarrow K_0(\drho)$. Since the Grothendieck groups are $\mathbb{Z}$-modules, we can use the structure theorem of finitely generated modules over principal ideal domains to conclude that $K_0(\drho) \simeq \mathbb{Z}$. 
\end{proof}
\begin{lem}\label{4b}\textnormal{ It is also true that $K_0(\gr \drho) \simeq K_0(\gr K[G/H])$. }
\end{lem}
\begin{proof} Analogous to the proof of \ref{gr} one can easly see that we have an injective and surjective morphism of filtered rings \[ k[x_0,x_0^{-1}][G/H] \hookrightarrow \gr \drho \rightarrow k[x_0,x_0^{-1}][G/H]\] and if we compose them we get the identity. So it shows that $K_0(k[x_0,x_0^{-1}][G/H])$ can be embedded into $K_0(\filg/\fillg)$. In the same way again as in lemma \ref{e} we can see that $K_0(\gr \drho)$ also can be embedded into $K_0(k[x_0,x_0^{-1}][G/H)$. The fact the we can do these is that $x_0$ is the principal symbol of the uniformizer element in $K$ so it commutes with every other element and the action of $G$ is trivial on it.
\end{proof}
\begin{rem}\textnormal{The only problem is that when a ring has infinite global dimension, it is not true that $K_0(R) \simeq G_0(R)$. As both $k[G/H]$ and $\gr \drho$ has most of the time infinite global dimension ($k$ characteristic $p$ and usually $p$ divides the cardinality of $G/H$), the Fundamental Theorem (Theorem \ref{fund})fails to give us any information regarding the classical Grothendieck group in which we are interested in, in fact it only provides us the following information: \[\mathbb{Z}^c \simeq G_0(\grr \drho) \simeq G_0(\gr \drho)\] and  \[G_0(k[x_0,x_0^{-1}][G/H]) \simeq G_0(k[G/H]) \] }
\end{rem} 
\noindent In the pro-$p$ case we automaticaly had in injective map $\mathbb{Z} \hookrightarrow K_0(\drho)$ and that was enough, but  we must follow another path, in order to prove our result in the general case.
\begin{rem}\textnormal{\label{xi} Recall that by  the localization theorem, we have the following exact sequence \[ \begin{CD}  K_0($p$-tors) @>\varphi>> G_0(\mathcal{O}_K[[G]]) @>\xi>> G_0(K[[G]])@>>> 0 \end{CD}\]
where $O_K$ is the ring of integers in $K$ and we have the same if we consider $\filg$, $\drho$ instead of $\mathcal{O}_K[[G]]$ and $K[[G]]$.}
\end{rem}
\section{Injectivity of $\xi$ }
\noindent  In this section we prove that the map $\xi$ defined in Remark \ref{xi} is also injective in the continuous distribution algebra case hence it is an isomorphism. We will call a  $R=\mathcal{O}[[G]]$-module strict $p$-torsion module if $Mp=0$.
\begin{lem}\label{61}\textnormal{Let $M$ be a strict $p$-torsion $R$-module. The image of $[M]$ with respect to $\varphi$ has finite order.}
\end{lem}
\begin{proof}The image of a  $p$-torsion module is itself since  the map $\varphi$ is induced by the natural inclusion $p-tors \subset \textnormal{mod}-R$. First we investigate the case when $M$ has global dimension $0$. If $M$ is free of rank one, we have a short exact sequence of $R$-modules \[ \begin{CD} 0 @>>> R @>\cdot p>> R @>\pi>> R/Rp @> >> 0 \end{CD}\] hence $[M]=[R/Rp] = [R] - [R] =0$. One can decude the same for any finitely generated free module. So the order of $M$ is $1$. \newline \noindent If $M$ is projective, we have $M \oplus Q \simeq F$. Since $\filg$ is $p$-adicaly complete,  we can take the idempotent lifting of $M$ to $R$ and denote it by $\overline{M}$. We know that it is projective and  have the following exact sequence of $R$-modules \[ \begin{CD} 0 @>>> \overline{M} @>\cdot p>> \overline{M} @>>> Coker @>>> 0 \end{CD}\] \noindent The cokernel and $M$ are isomorphic mod $p$ since we considered the idempotent lifting of $M$ with respect to the surjection $\pi:R \to R/Rp$. But it means that the classes are the same in $G_0(R)$, i.e. $[Coker]=[M]$. But the class  $[Coker] = [\overline{M}] - [\overline{M}] = 0$.
 \newline \noindent We call an $R$-module $p$-regular if $mp=0$ implies that $m=0$. Now we recall the following result:
\begin{thm}\label{pd}
\textnormal{Let $R$ be a ring, $x \in R$ a regular normal non-unit (regular means that it is neither left or right zero-divisor and normal means that $xR=Rx$), and suppose that $M$ is an $x$-regular $R$-module (analogously $x$-regular means that $mx=0 \Rightarrow m=0$). Then $\textnormal{pd}_{R/Rx} (M/Mx) \leq \textnormal{pd}_{R}(M)$. If moreover $R$ is noetherian, $M$ is finitely generated and $x \in J(R)$, then $\textnormal{pd}_{R/Rx} (M/Mx) = \textnormal{pd}_{R}(M)$.
}
\end{thm}
\begin{proof} It can be found in \cite{CM} Prop. $7.3.6.$
\end{proof}
\noindent We proceed by induction on the lenght of the minimal projective $R$-resolution of $M$. We saw in Section \ref{iwa} that $R/Rp = k[[G]]$ has finite global dimension. Let us suppose first that $n=1$ and  recall from Theorem $7.3.5.$ (i) in \cite{CM} that pd$_R(M)$= pd$_{R/Rp} +1$ hence $M$ is a projective $R$- modules. We saw that the image of the class of a projective $R/Rp$-module is 0. Let us suppose that the Lemma is true for $n-1$. For $n$ we have the following exact sequence \[ \begin{CD} 0 @>>> K @>>>F @>>> M @>>> 0 \end{CD}\] where $K$ is the first syzygy and $F$ is a free $R$-module. Since $Mp =0$ we see that $Fp \subseteq K$ hence we have the chain \[ F \supseteq K \supseteq pF \supseteq pK \supseteq p^2F \supseteq p^2K \supseteq \dots \] hence the following sequence of $R/Rp$-modules is exact \[ \begin{CD} 0 @>>> K/Fp @>>> F/Fp @>>> M @>>> 0 \end{CD}\] Now this is also an exact sequence of $R$-modules with the $R$-action induced by the surjection $\pi:R \to R/Rp$ (in this case the category of $p$-torsion modules is a Serre-subcategory in $\textnormal{mod}-R$ and the objects in the category of $R/Rp$-modules are naturaly $p$-torsion $R$-modules and these objects form an exact subcategory in $p$-tors). By assumption $\textnormal{pd}_R(M) =n$ and since $\textnormal{pd}_R(F) =0$ (moreover it is a free $R/Rp$-module), using Theorem $\ref{pd}$ one sees that $\textnormal{pd}_R(F/Fp) =1$. Now by the properties of the projective dimension function (for details see $7.1.6.$ in \cite{CM}) we see that $\textnormal{pd}_R(K/Fp) \leq n-1$ (except the case when $\textnormal{pd}_R(M) = 1$ which we have already handled). But $K/Fp$ is also a strict $p$-torsion module since it is an $R/Rp$-module, hence the induction applies. So the order of $[K/Fp]$ is finite, say $k$. Now summing the exact sequence above finitely many times we get the following  \[ \begin{CD} 0 @>>> \oplus K/Fp @>>> \oplus F/Fp @>>> \oplus M @>>> 0 \end{CD}\] By the definition of the Grothendieck group we know that $[\oplus_s N]=[N]^s = s[N]$ for any finitely generated $R$-module $N$. So $[M]^k = [F/Fp]^k - [K/pF]^k$. Now by Theorem \ref{pd} we know that $F/Fp$ is projective and it is a strict $p$-torsion module so the order of the image of  its class is $1$. Hence we are done, since $k \geq 1$.
\end{proof}
\begin{lem}\label{tor}\textnormal{ Let us assume that $M$ is a general $p$-torsion $R$-module. The order of the image of its class has finite order.}
\end{lem}
\begin{proof} Once again we recall from Remark \ref{xi} that we have the following exact sequence \[ \begin{CD}  K_0($p$-tors) @>\varphi>> G_0(\mathcal{O}_K[[G]]) @>\xi>> G_0(K[[G]])@>>> 0 \end{CD}\] Define $R : = \mathcal{O}_K[[G]]$ again. By Devissage Theorem, we can identify the first group with $G_0(R/Rp)$ since for any finitely generated $p$-torsion module $M$, there exist a positive integer $n$, such that $Mp^n = 0$ therefore there is a filtration  $M \supset Mp \supset Mp^2 \supset \dots \supset Mp^{n-1} \supset 0$ and all the quotients are $R/Rp$-modules. So the class $[M]$ is equal to the class $\sum[Mp^{i-1}/Mp^{i}]$. All the modules $Mp^{i-1}/Mp^{i}$ are finitely generated strict $p$-torsion modules. Hence by Lemma \ref{61} the classes of these modules in $G_0(\textnormal{mod}-R)$ have finite order. Since all the Grothendieck groups are Abelian, we have the following equality: \[([M])^k=k[M] =\sum k [Mp^{i-1}/Mp^{i}] = \sum ([Mp^{i-1}/Mp^{i}])^k\] Let $k_0$ be the smallest common multiple of the orders and let $k=k_0$. Using the equality above  we are done.
\end{proof}
\begin{pro}\textnormal{The Grothendieck group of $\mathcal{O}_K[[G]]$ is isomorphic to $\mathbb{Z}^c$.}
\end{pro}
\begin{proof}\textnormal{ The quotient ring $\mathcal{O}_K[[G]]/\textnormal{Jac}(\mathcal{O}_K[[G]])$ where Jac($\mathcal{O}_K[[G]])$ denotes the Jacobson radical is the same as $k[G/H]/\textnormal{Jac}(k[G/H])$ because the kernel of the augmentation map $\mathcal{O}_K[[G]] \to k[G/H]$ is a subset of the Jacobson radical of $\mathcal{O}_K[[G]]$. Therefore \[K_0(\mathcal{O}_K[[G]]) \simeq K_0(k[G/H])\] By Lemma \ref{k} we are done. }
\end{proof}
\noindent Using the Lemma \ref{tor} one gets the following:
\begin{pro}\label{kcomp}
\textnormal{ Let $G$ by a compact $p$-adic analytic group and assume in addition that it has no element of order $p$. Then $K_0(K[[G]]) \simeq \mathbb{Z}^c$.
}
\end{pro}
\begin{proof} By Theorem \ref{okok} and Remark \ref{global}, $G_0(\mathcal{O}_K[[G]]) \simeq \mathbb{Z}^c$ so it a free abelian group without torsion elements. By Lemma \ref{tor}  all the images with respect to $\varphi$ go to zero. This means that $\xi$ is also injective. The global dimension of $\mathcal{O}_K[[G]]$ is finite and $K[[G]]$ is just the localization of $O_K[[G]]$ hence its global dimension of bounded above by the global dimension of $\mathcal{O}_K[[G]]$. By Remark \ref{global} we are done.
\end{proof}
\subsection{Proof of Theorem \ref{main}}\noindent  Now by Theorem \ref{dim} and  Lemma \ref{filk} and Remark \ref{global} we have that  $K_0(\filg) \simeq G_0(\filg)=\mathbb{Z}^c$. Hence by the Localization sequence that we used may times we have a surjective map $\xi: \mathbb{Z}^c \to G_0(\drho)$. By Remark \ref{global} and Theorem \ref{dglob} we have that $K_0(\drho) \simeq G_0(\drho)$. Unfortunately we cannot use the proof we have in the continuous case since the quotient ring $\filg/\filg\cdot p$ has infinite global dimension. But we have Theorem $5.2.$ in \cite{S1} so by the faithfully flatness property we have a map  $K_0(K[[G]]) \hookrightarrow K_0(\drho)$ and it is injective. By Proposition \ref{kcomp} we have both an injective and a surjective map between $\mathbb{Z}^c$ and $K_0(\drho)$. These are all abelian groups hence $\mathbb{Z}$-modules. Using the structure theorem of finitely generated module over a PID we get an isomorphism $\mathbb{Z}^c \simeq K_0(\drho)$. 

\end{document}